\documentclass[a4paper]{amsart}

\usepackage[english]{babel}    
\usepackage[utf8]{inputenc}
\usepackage[T1]{fontenc}
\usepackage[margin=1.7in]{geometry}
\usepackage{amssymb, amsmath, amsfonts, mathrsfs, amsthm, amsbsy}
\usepackage{url}
\usepackage{tikz}
\usepackage{tikz-cd}
\usepackage{stmaryrd} 
\usepackage{enumitem}

\theoremstyle{plain}
\newtheorem{thm}{Theorem}[section]
\newtheorem{lem}[thm]{Lemma}
\newtheorem{prop}[thm]{Proposition}
\newtheorem{cor}[thm]{Corollary}

\theoremstyle{definition}
\newtheorem{defn}{Definition}[section]
\newtheorem{Conv}{Convention}[section]
\newtheorem*{Claim}{Claim}

\theoremstyle{remark}
\newtheorem{rem}{Remark}

\numberwithin{equation}{section}

\newcommand{\Z}{\mathbb{Z}}

\newcommand{\C}{\mathbb{C}}
\newcommand{\N}{\mathbb{N}}

\DeclareMathOperator{\Spec}{Spec}

\DeclareMathOperator{\val}{val}

\DeclareMathOperator{\Ri}{R}

\DeclareMathOperator{\Jac}{Jac}
\DeclareMathOperator{\Hi}{H}
\DeclareMathOperator{\rank}{rank}
\DeclareMathOperator{\ch}{ch}
\DeclareMathOperator{\Td}{Td}
\newcommand{\DR}{\mathrm{D\!R}}
\newcommand{\aj}{\bs{a}}

\newcommand{\cB}{\mathcal{B}}
\newcommand{\cM}{\mathcal{M}}
\newcommand{\cMm}{\smash{\overline{\mathcal{M}}}}%
\newcommand{\cJj}{\smash{\overline{\mathcal{J}}}}
\newcommand{\cJ}{\mathcal{J}}
\newcommand{\cI}{\mathcal{I}}
\newcommand{\tJj}{\smash{\widetilde{\mathcal{J}}}}
\newcommand{\sO}{\mathscr{O}}%
\newcommand{\bP}{\mathbb{P}}
\newcommand{\fZ}{\mathcal{Z}}

\newcommand{\mr}[1]{\mathrm{#1}}
\newcommand{\mc}[1]{\mathcal{#1}}
\newcommand{\mf}[1]{\mathfrak{#1}}

\newcommand{\bs}[1]{\boldsymbol{#1}}

\begin{document}

\title[Universal jacobian and the double ramification cycle]{Compactified universal jacobian and the double ramification cycle}
\author{Bashar Dudin}
\date{\today}
\address{Bashar Dudin\\
Centro de Matemática da Universidade de Coimbra\\
Apartado 3008\\
EC Santa Cruz \\
3001 - 501 Coimbra \\
Portugal}
\email{dudin@mat.uc.pt}
\thanks{The author was supported by the FCT research project \emph{Geometry of moduli spaces of curves and abelian varieties} EXPL/MAT-GEO/1168/2013}
\subjclass[2010]{14H10, 14H40, 14H60}
\maketitle

\begin{abstract}
  Using the compactified universal jacobian $\cJj_{g, n}$ over the moduli space of stable marked curves $\cMm_{g, n}$, we give an expression in terms of natural classes of the zero section of $\cJj_{g, n}$ in the (rational) Chow ring of $\cJj_{g, n}$. After extending variants of the Abel-Jacobi map to a locus containing curves of treelike type we give a formula for the pullback of the said zero section along these extensions. The same approach is also applied to recover known formulas for the pullback of theta divisors to the moduli space of marked stable curves.
\end{abstract}

\section{Introduction}

\subsection{The Eliashberg problem} 
Let $\underline{\tau} = (\tau_1, \ldots, \tau_n)$ be an integer valued non-zero $n$-tuple having $0$ sum entries. The double ramification cycle $[\DR]_{g, \underline{\tau}}$ in the moduli space of (smooth) genus $g$ $n$-marked curves $\cM_{g, n}$ is defined by the locus $\DR_{g, \underline{\tau}}$ of marked curves $(X, \underline{p})$, for $\underline{p} = (p_1, \ldots, p_n)$,  such that 
\begin{displaymath}
  \sO_X\Big(\sum_{i=1}^n \tau_ip_i\Big) \simeq \sO_X.
\end{displaymath}
This is precisely the locus of genus $g$ curves having a dominant map to the projective line, such that the fiber of $\infty$ is given by the $p_i$ points with positive $\tau_i$, the fiber of $0$ is given by the $p_i$ points with negative $\tau_i$ and ramification index at $p_i$ is $|\tau_i|$. The Eliashberg problem is to find compactifications of $\DR_{g, \underline{\tau}}$ in the moduli space of genus $g$ stable $n$-marked curves $\cMm_{g, n}$ where to write down the closure of $\DR_{g, \underline{\tau}}$ in terms of tautological classes of $\cMm_{g, n}$. 

There has been mainly two approaches to tackle the previous question, at least on partial compactifications of $\DR_{g, \underline{\tau}}$. One is to define an extension $[\DR]_{g, \underline{\tau}}^{GW}$ of $[\DR]_{g, \underline{\tau}}$ over the whole $\cMm_{g, n}$ as the pushforward of the virtual fundamental class of the moduli space of relative stable maps to a rubber $\bP^1$, by the forgetful map to $\cMm_{g, n}$. This is the point of view of \cite{CavalieriMarcusWisePolynomialfam}, it gives rise to a formula for $[\DR]_{g, \underline{\tau}}^{GW}$ over the locus of curves having rational tails. The other approach, on which we build up, is to look for extensions of the Abel-Jacobi map $\underline{\aj}$ from $\cM_{g, n}$ to the universal family $\mc{X}_g$ over the moduli space of polarized abelian varieties $\mathcal{A}_g$, sending a curve $(X, \underline{p})$ to $(\mathrm{Jac}(X), \sO_X(\sum_i \tau_ip_i))$, and compute the pullback of the zero section of $\mathcal{X}_g$ along these extensions. By using the extension of $\underline{\aj}$ to the locus $\cM_{g, n}^{ct}$ of curves of compact type R.~Hain computed in \cite{HainNormalFunctions} the class of an extension of $[\DR]_{g, \underline{\tau}}$ in terms of tautological classes over $\cM_{g, n}^{ct}$. By \cite{MarcusWiseStableMapsRelativeJacobian} this class happens to be equal to $[\DR]_{g, \underline{\tau}}^{GW}$ when restricted to curves of compact type. The starting point of Hain's formula is the elegant relation
\begin{equation}
  \label{eq:Theta}
  [\fZ_g] = \frac{[T]^g}{g!}
\end{equation}
expressing the zero cycle of $\mathcal{X}_g$ in terms of the universal symmetric theta divisor $T$ trivialized along the zero section. The pullback of $T$ to $\cM_{g, n}^{ct}$ is then given by 
\begin{equation}
  \label{eq:Theta2}
  \underline{\aj}^*[T] = -\frac{1}{4}\sum_{\substack{0 \leq h \leq g \\ A \subset \{1, \ldots, n\} \\ 1\leq h + |A| \leq g+n-1}} \Big(\sum_{i \in A}\tau_i\Big)^2\delta_{h, A},
\end{equation}
where $\delta_{0, i} = - \psi_i$ and $\delta_{h, A} = \delta_{g-h, A^c}$ are boundary divisors for $(h, A) \neq (0, \{i\})$ or $(g, \{1, \ldots, n\}-\{i\})$. Formula \ref{eq:Theta2} was later proven using simpler techniques in \cite{GrushZakDRC}. Formula \ref{eq:Theta} was extended by S.~Grushevsky and D.~Zakharov in \cite{GrushZakzerosection} to the partial compactification of $\mathcal{A}_g$ by rank $1$ semi-abelian varieties. This gave rise to an extension of the formula of $\underline{\aj}^*[\fZ_g]$ to the case of stable marked curves having at most one non-separating node. Recently A.~Pixton conjectured a formula for $[\DR]_{g, \underline{\tau}}^{GW}$ on $\cMm_{g, n}$.

The previous questions were considered by Hain in a more general setting. Let $k$ be an integer and $\underline{\tau}$ an integer valued $n$-tuple, such that $\sum_i\tau_i = k(2g-2)$. Write $\DR_{g, \underline{\tau}}^k$ for the locus of curves $(X, \underline{p})$ in $\cM_{g, n}$ satisfying  
\begin{displaymath}
  \sO_{X}\Big(\sum_{i=1}^n \tau_ip_i\Big) = \omega_X^{\otimes k}
\end{displaymath}
 where $\omega_X$ is the dualizing sheaf of $X$--in the smooth case it is just the sheaf of regular differential forms. From a modular point of view this is about looking at curves together with (rational) multiples of canonical divisors having a given specific form. Hain gives equivalent formulas to \ref{eq:Theta2} for any $k$ over $\cM_{g, n}^{ct}$ and Pixton's conjectures extend to these cases.

\subsection{Content of the present paper}
Rather than looking at the whole moduli space $\mathcal{A}_g$, we limit ourselves to the case of the universal jacobian $\cJ_{g, n}$ over $\cM_{g, n}$. A point in $\cJ_{g, n}$ is simply a tuple $(X, \underline{p}, L)$ where $(X, \underline{p})$ is a smooth marked curve and $L$ is a degree $0$ line bundle on $X$. For each $k$ and $\underline{\tau}$ the Abel-Jacobi map $\aj_k : \cM_{g, n} \rightarrow \cJ_{g, n}$ sending $(X, \underline{p})$ on $(X, \underline{p}, \sO_X(\sum_i\tau_ip_i)\otimes\omega_X^{-k})$ gives a section of the forgetful map $\epsilon_n : \cJ_{g, n} \rightarrow \cM_{g, n}$.

A standard way to look for compactifications of $\cJ_{g, n}$ is by allowing points $(X, \underline{p}, L)$ where $L$ is a simple torsion-free sheaf of rank $1$ on the stable marked curve $(X, \underline{p})$. To get a proper stack one has to make assumptions on the multidegree of $L$. This is done by choosing a polarisation on $\cMm_{g, n}$, imposing a stability condition on each torsion-free sheaf $L$, see section \ref{sec:CompactifiedUniversalJacobian} for more details. In our case, the natural polarisation is the one giving an extension of the zero section $\fZ_{g, n}$ to the corresponding compactification $\cJj_{g, n}$. In section \ref{sec:TheZeroCycle} we describe $\fZ_{g, n}$ as the degeneracy locus of the connection morphism $\phi_*(\mf{L}\otimes \sO_{D_{n}}) \rightarrow \Ri^1\phi_*\mf{L}(-D_n)$ where $\phi : \cJj_{g, n\mid 1} \rightarrow \cJj_{g, n}$ is the universal curve over $\cJj_{g, n}$, $\mf{L}$ is the universal quasi-stable torsion-free sheaf over $\cJj_{g, n}$ (a sheaf over $\cJj_{g, n\mid 1}$ in fact) and $D_n$ the $n$-th marked section. Using Thom-Porteous formula we get the relation
\begin{equation}
  \tag{\ref{eq:Z}}
  [\fZ_{g, n}] = \bigg\{\exp\Big(\sum_{s\geq 1} (-1)^s(s-1)!\big\{\phi_*(\ch(\mf{L})\Td^\vee(\Omega_\phi))\big\}_s\Big)\bigg\}_g
\end{equation}
where $\Omega_{\phi}$ is the relative sheaf of differential forms over $\cJj_{g, n}$ and $\{\bullet\}_\ell$ the degree $\ell$ part of $\bullet$. 

The usefulness of formula \ref{eq:Z} lies in the fact that it is straightforward to pullback classes obtained by pushing forward Chern classes of $\mf{L}$ and $\Omega_\phi$. When $\DR_{g, \underline{\tau}}^k$ is of expected codimension $g$, to give an extension of its cycle is thus only about extending $\aj_k$. This is for instance the case when $k = 0$ or $k=1$ and one of the $\tau_i$s is negative, an account of these facts can be found in \cite{2015arXiv150807940F}. The pullback of the zero section along extensions of $\aj_k$ does always make sense though. 

Let $\pi : \cMm_{g, n+1} \rightarrow \cMm_{g, n}$ be the universal curve over $\cMm_{g, n}$ and $D_i$ the section of $\pi$ corresponding to the $i$-th marked point. To extend $\aj_k$ means giving a quasi-stable line bundle over a locus of $\cMm_{g, n}$ whose restriction to $\cM_{g, n}$ is given by $\sO(\sum_i\tau_iD_i)$. Let $\cMm_{g, n}^{k, \underline{\tau}}$ be the union of the locus of curves of treelike type and the one of curves $(X, \underline{p})$ satisfying, for each union of irreducible components $Z \subset X$
\begin{equation}
  \tag{\ref{eq:balanced}}
  \sum_{p_i \in Z} \tau_i \geq k\deg(\omega_{X\mid Z}) - \frac{|Z\cap \overline{X - Z}|}{2},
\end{equation}
with strict inequality whenever $p_1 \in Z$ and $\emptyset \subsetneq Z \subsetneq X$. The complement of $\cMm_{g, n}^{k, \underline{\tau}}$ in $\cMm_{g, n}$ is the closure of the locus of curves having topological type a two vertex loopless graph, which doesn't satisfy \ref{eq:balanced}. By twisting $\sO(\sum_i\tau_iD_i)$ we show in section \ref{sec:PartialExtensionofAbel-JacobiMaps} that the line bundle 
\begin{equation}
  \tag{\ref{eq:qk}}
  \mc{L}(\underline{\tau}, k) = \sO\Big(\sum_{i=1}^n \tau_iD_i\Big)\otimes \omega^{-k}\otimes\sO\Big(\sum_{\substack{0 \leq h \leq \lfloor g/2 \rfloor \\ A \subset \{1, \ldots, n\} \\ 2 \leq |A| + h \leq g+ n-2 }} \big[k(1-2h)+\sum_{i \in A} \tau_i\big]\delta_{h, A\cup\{n+1\}}\Big).
\end{equation}
is a quasi-stable line bundle over $\cMm_{g, n}^{k, \underline{\tau}}$. The Abel-Jacobi map does thus extend to maps $\bar{\aj}_k$ defined over $\cMm_{g, n}^{k, \underline{\tau}}$. Obstructions to extend yet further ${\aj}_k$ maps are shortly discussed in section \ref{sec:PartialExtensionofAbel-JacobiMaps}. Pulling back \ref{eq:Z} along $\bar{\aj}_k$ we get the formula
\begin{equation}
  \tag{\ref{eq:finale}}
  \bar{\aj}_k^*[\fZ_{g, n}] = \bigg\{\exp\Big(\sum_{s\geq 1} (-1)^s(s-1)!\big\{\pi_*\big(\ch(\mc{L}(\underline{\tau}, k)\big)\Td^\vee(\Omega_\pi))\big\}_s\Big)\bigg\}_g.
\end{equation}
It agrees with Hain's formula for curves of compact type but we don't have a direct proof of this fact. The right hand side of \ref{eq:finale} is an explicit degree $g$ polynomial in tautological classes on $\cMm_{g, n}$. It is unfortunately much less transparent than the formula given by Hain and comparison with the extension given by Grushevsky and Zakharov is not clear. 

The starting point of the previous result lies in the fact we're able to write down the zero cycle of the compactified jacobian in terms of its universal sheaf, as well as sheaves coming from $\cMm_{g, n}$. The pullback along extensions of Abel-Jacobi maps is then straightforward. Following the same line of thought, we write down the class of the universal theta divisor $[\Theta]$ (trivialized along the zero section) in terms of $\mf{L}$, in this case we have that 
\begin{equation}
  \tag{\ref{eq:ThetaL}}
  [\Theta] = - \phi_*\left(\frac{c_1(\mf{L})^2}{2}\right).
\end{equation}
A simple computation in section \ref{sec:AFormulaForPullbacks} gives yet another way of computing the pullback of $[\Theta]$ along $\bar{\aj}_k$.

Lastly, in section \ref{subsec:Thetag}, we briefly explain how to adapt our strategy to compute the pullback of the universal theta divisor of the degree $g-1$ compactified jacobian along the Abel-Jacobi map. We get back the class computed in \cite{GrushZakDRC} up to the generic vanishing of the theta function along the locus of irreducible curves in $\cMm_{g, n}$. This also recovers the formula given by F.~M\"uller in \cite{MR3092284}, which computes the closure in $\cMm_{g, n}$ of the locus 
\begin{displaymath}
\mc{D}_g = \Big\{ (X, \underline{p}) \mid h^0\big(\sO_X(\sum_i \tau_i p_i)\big) \geq 1\Big\} 
\end{displaymath}
for smooth marked curves $X$, $\sum_i\tau_i = g-1$ and at least one $\tau_i$ is negative. 

After this paper was first posted to the arXiv, the author was made aware of related work of Jesse Kass and Nicola Pagani. In \cite{2015arXiv150703564K}, Kass and Pagani compute the pullback of the theta divisor of certain compactified Jacobians that are constructed in loc. cit. For a certain choice of stability parameter, they compute the pullback of the theta divisor to be the class in \ref{eq:Thetagfinal}. The author posted his paper to the arXiv on May 12, 2015. Kass and Pagani posted their preprint to the arXiv on July 13, 2015 and first publicly presented their work in a seminar at the University of Liverpool on March 10, 2015. 

\subsection*{Acknowledgements}
I'm grateful to Margarida Melo and Filippo Viviani for presenting me the questions studied here, for many useful discussions and for reading and commenting on early versions of the present paper.

\section{Compactified universal jacobian over $\cMm_{g, n}$}
\label{sec:CompactifiedUniversalJacobian}

We review the definition of compactified jacobians following \cite{EstevesCompactifyingJacobian}. Points of any such compactification are given by stable marked curves endowed with simple rank $1$ torsion-free sheaves satisfying a semi-stability condition. The first appearance of such compactifications is due to R.~Pandharipande in \cite{PandhSlope} for slope semi-stability, in the more general case of vector bundles. This last point of view, in the case of line bundles, is equivalent to the original approach by L.~Caporaso in \cite{Caporaso94compactification}. Both were compactifications of the universal jacobian over the moduli space of stable curves $\cMm_{g}$. The corresponding construction for marked stable curves was studied by M.~Melo in \cite{MeloCompactifiedPicardStacksMarked}. The equivalence between Pandharipande's and Caporaso's approaches were first proven in \cite{PandhSlope}. An extension to the marked case and more general stability conditions can be found in \cite{EstevesPaciniSemistablemodifications}. In the case at hand--the case of compactified universal jacobians--our main reference is \cite{MeloCompactificationsofJacobians}.

\begin{Conv}
  All our schemes are schemes over $\C$. We fix integers $d \in \Z$, $g \in \N$ and $n\in \N^*$ such that $2g+2+n > 0$. 
\end{Conv}

Let $\tJj_{g, n}^d$ be the stackification of the prestack whose $S$-sections are given by couples $(f : \mc{X} \rightarrow S, \underline{D}, \mc{L})$ consisting of a relative stable curve $f$ of genus $g$, disjoint \'etale relative Cartier divisors $\underline{D} = (D_1, \ldots, D_n)$ in the smooth locus of $f$ and a coherent sheaf $\mc{L}$ over $\mc{X}$ which is flat over $S$ and has simple torsion-free degree $d$ geometric fibers. The Cartier divisors $\underline{D}$ are called markings or marked points of $f$ and are assumed to be defined by images of sections of $f$. Morphisms from $(f : \mc{X} \rightarrow S, \underline{D}, \mc{L})$ on $(\psi: \mc{Y} \rightarrow T, \underline{E}, \mc{M})$ over $\alpha : S \rightarrow T$ are given by $S$-morphisms $\beta : \mc{X} \rightarrow \alpha^*\mc{Y}$ commuting with the markings and for which there is an invertible sheaf $\mc{N}$ on $S$ and an isomorphism $\mu : \mc{L} \simeq \beta^*\mc{M}\otimes f^*\mc{N}$.       
\begin{thm}[{\cite{MeloCompactificationsofJacobians}}]
  \label{thmMelo1}
  The stack $\tJj_{g, n}^d$ is a smooth irreducible DM-stack of dimension $4g-3+n$. It is representable over $\cMm_{g, n}$ and satisfies the existence part of the relative criterion for properness. 
\end{thm}
The previous theorem says that to compactify the open substack $\cJ_{g, n}^d$, corresponding to smooth underlying curves, one has to look for separated universally closed substacks of $\tJj_{g, n}^d$. 

  Let $X$ be a stable marked curve of genus $g$ over $\C$ or, more generally, over an algebraically closed field. By a subcurve of $X$ we shall always mean a \emph{closed} subcurve of $X$. A polarization of degree $d$ on $X$ is a locally free $\sO_X$-module $P$ of rank $r> 0$ and degree $r(d+1-g)$. Given any subcurve $Y \subset X$ we write 
\begin{displaymath}
  q_Y(P) = \frac{\deg P_{\mid Y}}{r} + \frac{\deg \omega_{{X} \mid Y}}{2}
\end{displaymath}
where $\omega_{X}$ is the canonical sheaf on $X$. 
\begin{defn}
  \label{defn:semistability}
  Let $P$ be a polarization a degree $d$ on $X$ and $x \in X$ a smooth point. Given a closed subcurve $Y$ of $X$ we write $\kappa_Y$ for the number of points in the intersection of $Y$ with the closure of its complement in $X$. Given a line bundle $L$ on $X$ we write $\deg_YL$ for the degree of $L_{\mid Y}$ modulo torsion. A torsion-free sheaf $L$ of rank $1$ on $X$ of degree $d$ is said to be 
\begin{enumerate}
  \item $P$-semistable if for all non-empty subcurves $Y \subsetneq X$, $\deg_Y L  \geq q_Y - \frac{\kappa_Y}{2}$
  \item $P$-stable if inequalities in 1. are strict
  \item $(P, x)$-quasistable (q-stable for short) if (1) is verified and is strict whenever $x \in Y$.
\end{enumerate}
\end{defn}
\noindent By a polarization on $\cMm_{g, n}$ we understand a polarization on the universal curve over $\cMm_{g,n}$. This means a locally free sheaf $\mc{P}$ on $\cMm_{g, n+1}$ of rank $r$ and degree $r(d+1-g)$. 

Fix a polarization $\mc{P}$ on $\cMm_{g, n}$ and a section $\tau$ of $\cMm_{g, n+1}\rightarrow \cMm_{g, n}$. Write $\cJj_{g, n}^{d, \mc{P}, ss}$, $\cJj_{g, n}^{d, \mc{P}, s}$ and $\cJj_{g, n}^{d, \mc{P}, \tau}$ for the substacks of $\tJj_{g, n}^d$ given by triples $(f : \mc{X} \rightarrow S, \underline{D}, \mc{L})$ for which $\mc{L}$ has respectively $P$-semistable, $P$-stable and $(P, \tau)$-quasistable simple torsion-free geometric fibers of degree $d$. 
\begin{thm}[{\cite{MeloCompactificationsofJacobians}}]
  \label{thmMelo2}
$\cJj_{g, n}^{d, \bullet}$ are smooth irreducible DM-stacks of dimension $4g-3+n$ that are representable and of finite type over $\cMm_{g, n}$. Furthermore
\begin{enumerate}
  \item $\cJj_{g, n}^{d, \mc{P}, ss}$ is universally closed
  \item $\cJj_{g, n}^{d, \mc{P}, s}$ is separated
  \item $\cJj_{g, n}^{d, \mc{P}, \tau}$ is proper and has a projective coarse moduli space.
\end{enumerate}
\end{thm}
In the following, we will mainly be interested in the compactification $\cJj_{g, n}$ of $\cJ_{g, n}$ given by the polarization $\omega^{-1}\oplus \sO$ and q-stable torsion-free rank 1 sheaves. For clarity the section $\tau$ defining q-stability is $D_1$, the relative étale Cartier divisor defined by the first marking. We limit ourselves to this case from now on.

Let $\cJj_{g, n \mid 1}$ be the stack whose $S$-sections are given by sections $(f : \mc{X} \rightarrow S, \underline{D}, D_{n+1}, \mc{L})$ where $D_{n+1}$ is an extra section. There are no assumptions on $D_{n+1}$; it can intersect other marked sections or the singular locus of $f$. The universal family over $\cJj_{g, n}$ is given by the map $\phi : \cJj_{g, n \mid 1} \rightarrow \cJj_{g, n}$ forgetting the extra section. It comes with a universal torsion free sheaf $\mf{L}$ of rank $1$ given over an $S$-section $(f:\mc{X} \rightarrow S, \underline{D}, D_{n+1}, \mc{L})$ by the pullback of $\mc{L}$ along $D_{n+1}$. One has also another forgetful map $\epsilon_{n \mid 1} : \cJj_{g, n \mid 1} \rightarrow \cMm_{g, n+1}$ sending $(f : \mc{X} \rightarrow S, \underline{D}, D_{n+1}, \mc{L})$ on $(f : \mc{X} \rightarrow S, \underline{D}, D_{n+1})$. The stack of stable $n$-marked curves having an extra section is isomorphic to the one of $n+1$-marked stable curves. We shall constantly make this identification. This discussion fits into the picture
\begin{equation}
\label{eq:ncartesian}
\begin{tikzpicture}[>=latex, baseline=(current  bounding  box.center)]
    \matrix (m) 
    [matrix of math nodes, row sep=2.5em, column sep=2.5em, text
    height=1ex, text depth=0.25ex]  
    { \mf{L} & \cJj_{g, n \mid 1} & \cJj_{g, n} \\
      & \cMm_{g, n+1} & \cMm_{g, n} \\}; 
    \path[-, font=\scriptsize]
    (m-1-1) edge (m-1-2);
    \path[->,font=\scriptsize]  
    (m-1-2) edge node[auto] {$\phi$} (m-1-3)
    (m-1-2) edge node[auto] {$\epsilon_{n \mid 1}$} (m-2-2)
    (m-1-3) edge node[auto] {$\epsilon_{n}$} (m-2-3)
    (m-2-2) edge node[auto] {$\pi$} (m-2-3);
  \end{tikzpicture}
\end{equation}
where the square diagram is commutative but generally not cartesian.

\section{The zero cycle of $\cJj_{g,n}$ and class of the theta divisor}
\label{sec:TheZeroCycle}

\subsection{The zero cycle of $\cJj_{g, n}$. } 
A line bundle $L$ on a stable curve $X$ is semistable for $\omega^{-1}\oplus \sO$ if and only if for each proper subcurve $Y \subset X$ the inequality
\begin{displaymath}
  \deg_YL \geq -\frac{\kappa_Y}{2}
\end{displaymath}
holds. In particular the trivial line bundle $\sO_X$ is stable and thus q-stable. The zero section of the universal Jacobian over $\cM_{g, n}$ does therefore extend to $\cJj_{g, n}$, its image is simply denoted by $\fZ_{g,n}$. It is a smooth irreducible substack of $\cJj_{g, n}$ of codimension $g$, its cycle in the Chow ring of $\cJj_{g, n}$ is written $[\fZ_{g, n}]$. We will be expressing this cycle in terms of ``natural'' cycles on $\cJj_{g, n}$. Namely the Chern characters of the universal torsion-free sheaf $\mf{L}$, the relative dualizing sheaf of $\cJj_{g, n \mid 1} \rightarrow \cJj_{g, n}$ and boundary classes. We first start by a lemma on the space of global sections of $L$.
\begin{lem}
  For semistable $L$ we have that $h^0(L) \leq 1$ with equality if and only if 
$L$ is locally free of $\underline{0}$ multidegree. 
\end{lem}
\begin{proof}
  Let $\nu : X^\nu \rightarrow X$ be the normalisation of $X$. The natural map $L \rightarrow \nu_*\nu^*L$ induces an injective map $\Hi^0(L) \rightarrow \Hi^0(\nu^*L)$. Recall that 
  \begin{displaymath}
    \deg(L) = \deg(\nu^*L) + t_L
  \end{displaymath}
where $t_L$ is the number of nodal points of $X$ where $L$ is not free. In particular, when $L$ is not free the total degree of $\nu^*L$ is negative (by definition the total degree of $L$ is $0$). 

There is nothing to prove when $X$ has only one irreducible component. Assume that $X$ has at least two such components and $L$ is not of multidegree $\underline{0}$. Notice that this is automatically the case if $L$ is not free. One can write $X = Z\cup Z^c$ for $Z$ an irreducible component of $X$ such that $L_{\mid Z}$ has negative degree and $Z^c$ is the closure of the complement of $Z$ in $X$. Let $X_j$ for $j=1, \ldots, m$ be the connected components of $Z^c$ and write $E= Z\cap Z^c$, $E_j = E\cap X_j$. We have that
\begin{displaymath}
\Hi^0(X, L) =   \Hi^0\big(Z^c, L(\sum_{x \in E} -x)\big) = \bigoplus_{j} \Hi^0\big(X_j, L(\sum_{x \in E_j}-x)\big).
\end{displaymath}
Now the semistability condition says
\begin{displaymath}
  \deg_{X_j}L \leq \frac{|E_j|}{2}.
\end{displaymath}
The connectedness of $X$ insures $L(\sum_{x \in E_j}-x)$ has negative degree on $X_j$. Using iteratively semistability inequalities we get that $h^0(L)=0$. We've thus shown that if $L$ is not of multidegree $\underline{0}$ then $h^0(L)=0$. Since $L$ cannot be of multidegree $\underline{0}$ if non-free we have just shown our claim.  
\end{proof}
\begin{cor}
  For semistable $L$ on $X$, $h^0(L) \neq 0$ if and only if $L\simeq \sO_X$.
\end{cor}
\begin{proof}
  By the previous lemma $L$ has to be of multidegree $\underline{0}$. If a global section of $L$ vanishes at some point it vanishes on the irreducible component containing it. Since $X$ is connected it has to identically vanish on $X$. When $h^0(L) \neq 0$ there is thus a section that doesn't vanish anywhere on $X$. 
\end{proof}
The substack of $\cJj_{g, n}$ given by the zero section of $\cJj_{g, n} \rightarrow \cMm_{g, n}$ is formally defined by $S$-points $(f : \mc{X} \rightarrow S, \underline{D}, \mc{L})$ in $\cJj_{g, n}$ such that $\mc{L} \simeq f^*\mc{M}$ where $\mc{M}$ is an invertible sheaf on $S$. By base change in cohomology this is equivalently given by the $S$-points in $\cJj_{g, n}$ such that $\mc{L}_s \simeq \sO_{\mc{X}_s}$ for every fiber over $s \in S$. In the light of the previous lemma $\fZ_{g, n}$ is thus given by $S$-points $(f: \mc{X} \rightarrow S, \underline{D}, \mc{L})$ such that $h^0(\mc{L}_s) \geq 1$ for every $s \in S$. 

We describe $\fZ_{g, n}$ as the degeneracy locus of a map between vector bundles on $\cJj_{g, n}$. Let $(f : \mc{X} \rightarrow S, \underline{D}, \mc{L})$ be an $S$-point in $\cJj_{g, n}$. Pick one of the marked sections, say $D_n$, and consider the short exact sequence of sheaves over $\mc{X}$ 
\begin{displaymath}
  \begin{tikzpicture}[>=latex, baseline=(current  bounding  box.center)]
    \matrix (m) 
    [matrix of math nodes, row sep=2.5em, column sep=2.5em, text
    height=1ex, text depth=0.25ex]  
    { 0 & \mc{L}(-D_n) & \mc{L} & \sO_{D_n}\otimes \mc{L} & 0 \\}; 
    \path[->,font=\scriptsize]  
    (m-1-1) edge (m-1-2)
    (m-1-2) edge (m-1-3)
    (m-1-3) edge (m-1-4)
    (m-1-4) edge (m-1-5);
  \end{tikzpicture}
\end{displaymath}
 Its direct image by $f$ gives the long exact sequence 
 \begin{displaymath}
  \begin{tikzpicture}[>=latex, baseline=(current  bounding  box.center)]
    \matrix (m) 
    [matrix of math nodes, row sep=2.5em, column sep=2.5em, text
    height=1ex, text depth=0.25ex]  
    { 0 & f_*\mc{L} & \sO_{D_n}\otimes \mc{L} & \Ri^1\!f_*\mc{L}(-D_n)  & \Ri^1\!f_*\mc{L} & 0 \\}; 
    \path[->,font=\scriptsize]  
    (m-1-1) edge (m-1-2)
    (m-1-2) edge (m-1-3)
    (m-1-3) edge node[auto] {$\nabla$} (m-1-4)
    (m-1-4) edge (m-1-5)
    (m-1-5) edge (m-1-6);
  \end{tikzpicture}
 \end{displaymath}
The map $\nabla$ is a map between two locally free bundles over $S$. The source is of rank $1$ and the target of rank $g$. The degeneracy locus of $\nabla$ is defined by points $s \in S$ such that $\rank(\nabla\otimes k(s))\leq 0$. Because of commutation to base change of both source and target of $\nabla$ the tensor product $\nabla_s = \nabla\otimes k(s)$ is the one appearing in the long exact sequence in cohomology
\begin{displaymath}
  \begin{tikzpicture}[>=latex, baseline=(current  bounding  box.center)]
    \matrix (m) 
    [matrix of math nodes, row sep=2.5em, column sep=2.3em, text
    height=1ex, text depth=0.25ex]  
    { 0 & \Hi^0(\mc{L}_s) & \Hi^0(\sO_{D_n}\otimes \mc{L}_s) & \Hi^1(\mc{L}_s(-D_n))  & \Hi^1(\mc{L}_s) & 0 \\}; 
    \path[->,font=\scriptsize]  
    (m-1-1) edge (m-1-2)
    (m-1-2) edge (m-1-3)
    (m-1-3) edge node[auto] {$\nabla_s$} (m-1-4)
    (m-1-4) edge (m-1-5)
    (m-1-5) edge (m-1-6);
  \end{tikzpicture}
\end{displaymath}
The degeneracy locus of $\nabla$ is thus defined by the set of points $s \in S$ such that $h^0(\mc{L}_s) \geq 1$ and this is set theoretically $\fZ_{g, n}$. 
\begin{thm}
 Let $\phi : \cJj_{g, n \mid 1} \rightarrow \cJj_{g, n}$ be the universal curve over $\cJj_{g, n}$ and $\mf{L}$ the universal torsion-free sheaf over $\cJj_{g, n \mid 1}$. The zero cycle $[\fZ_{g, n}]$ is given by
\begin{equation}
\label{eq:Z}
[\fZ_{g, n}] = \bigg\{\exp\Big(\sum_{s\geq 1} (-1)^s(s-1)!\big\{\phi_*(\ch(\mf{L})\Td^\vee(\Omega_\phi))\big\}_s\Big)\bigg\}_g
\end{equation}
where $\Omega_\phi$ is the sheaf of relative differentials of $\cJj_{g, n \mid 1}$ over $\cJj_{g, n}$ and $\{\bullet\}_\ell$ is the degree $\ell$ part of $\bullet$.
\end{thm}
\begin{proof}
   The zero locus $\fZ_{g, n}$ is set theoretically equal to the degeneracy locus $\mathcal{D}_{\nabla}$ of $\nabla$. Since $\cJj_{g, n}$ is smooth and $\fZ_{g, n}$ of the expected codimension $g$ we get the relation
  \begin{displaymath}
    [\mathcal{D}_{\nabla}] = c_g\big(\Ri^1\!\phi_*\mf{L}(-D_n) - \phi_*(\sO_{D_n}\otimes \mf{L})\big)
  \end{displaymath}
out of the Thom-Porteous formula \cite{Fulton}. The standard relation between Chern classes and Chern characters gives
\begin{displaymath}
  [\mathcal{D}_{\nabla}] = \Big\{ \exp\Big(\sum_{s \geq 1}(-1)^{s-1}(s-1)!\Big[\ch_s\big(\Ri^1\!\phi_*\mf{L}(-D_n)\big)-\ch_s\big(\phi_*\sO_{D_n}\otimes \mf{L}\big)\Big]\Big)\Big\}_g
\end{displaymath}
The exact sequence 
\begin{displaymath}
  \begin{tikzpicture}[>=latex, baseline=(current  bounding  box.center)]
    \matrix (m) 
    [matrix of math nodes, row sep=2.5em, column sep=2.5em, text
    height=1ex, text depth=0.25ex]  
    { 0 & \phi_*\mc{L} & \phi_*(\sO_{D_n}\otimes \mc{L}) & \Ri^1\!\phi_*\mc{L}(-D_n)  & \Ri^1\!\phi_*\mc{L} & 0 \\}; 
    \path[->,font=\scriptsize]  
    (m-1-1) edge (m-1-2)
    (m-1-2) edge (m-1-3)
    (m-1-3) edge (m-1-4)
    (m-1-4) edge (m-1-5)
    (m-1-5) edge (m-1-6);
  \end{tikzpicture}
\end{displaymath}
and the additivity of the Chern character in exact sequences says
\begin{equation}
  \label{eq:Dnabla}
  [\mathcal{D}_{\nabla}] =  \Big\{ \exp\Big(\sum_{s \geq 1}(-1)^s(s-1)!\ch_s(\phi_!\mf{L})\Big)\Big\}_g.
\end{equation}
Using GRR formula we get the right hand side of \ref{eq:Z}.

Because $\fZ_{g, n}$ is irreducible, $[\fZ_{g,n}]$ and $[\mathcal{D}_{\nabla}]$ are proportional. To check they are equal it is enough to check their pullbacks along the zero section $\bs{z} : \cMm_{g, n} \rightarrow \cJj_{g, n}$ are the same. Let $\pi : \cMm_{g, n+1} \rightarrow \cMm_{g, n}$ be the universal curve over $\cMm_{g, n}$
\begin{Claim}
  The pullbacks $\bs{z}^*[\mathcal{D}_{\nabla}]$ and $\bs{z}^*[\fZ_{g, n}]$ are equal to $c_g(\Ri^1\!\pi_*\sO)$.
\end{Claim}
\noindent We start by computing $\bs{z}^*[\mathcal{D}_{\nabla}]$. The zero section extends naturally to a section $\hat{\bs{z}}$ on the universal curves by sending $(f : \mc{X} \rightarrow S, \underline{D}, D_{n+1})$ on $(f : \mc{X}\rightarrow S, \underline{D}, D_{n+1}, \sO_X)$. The resulting diagram 
\begin{displaymath}
\begin{tikzpicture}[>=latex, baseline=(current  bounding  box.center)]
    \matrix (m) 
    [matrix of math nodes, row sep=2.5em, column sep=2.5em, text
    height=1ex, text depth=0.25ex]  
    { \cJj_{g, n \mid 1} & \cMm_{g, n+1} \\
      \cJj_{g, n} & \cMm_{g, n} \\}; 
    \path[->,font=\scriptsize]  
    (m-1-1) edge node[left] {$\phi$} (m-2-1)
    (m-1-2) edge node[above] {$\hat{\bs{z}}$} (m-1-1)
    (m-1-2) edge node[auto] {$\pi$} (m-2-2)
    (m-2-2) edge node[auto] {$\bs{z}$} (m-2-1);
  \end{tikzpicture}
\end{displaymath}
is cartesian and on the level of Chow rings $\bs{z}^*\phi_* = \phi_*\hat{\bs{z}^*}$. This is a straightforward check; a detailed account of it can be found in section \ref{sec:AFormulaForPullbacks}. Using these facts we get from \ref{eq:Dnabla} that
\begin{displaymath}
    [\mathcal{D}_{\nabla}] = \Big\{ \exp\Big(\sum_{s \geq 1}(-1)^{s-1}(s-1)!\Big[\ch_s\big(\Ri^1\!\pi_*\sO\big)-\ch_s\big(\pi_*\sO\big)\Big]\Big)\Big\}_g = c_g\big(\Ri^1\!\pi_*\sO\big).
\end{displaymath}
To compute the pullback of $[\fZ_{g, n}]$ recall $\bs{z}^*[\fZ_{g, n}]$ is equally given by $\epsilon_{n *}\big([\fZ_{g, n}]^2\big)$ where $\epsilon_n : \cJj_{g, n} \rightarrow \cMm_{g, n}$ is the natural forgetful map. The latter class is then given in terms of the normal bundle along $\fZ_{g, n}$ by
\begin{displaymath}
  \bs{z}^*[\fZ_{g, n}] = \epsilon_{n *}\big(c_g(N_{\fZ_{g, n}})\cap [\fZ_{g, n}]\big). 
\end{displaymath}
On the level of tangent spaces the section $\bs{z}$ sends a first order deformation of a marked curve $(X, \underline{p})$ on the corresponding first order deformation together with the trivial bundle of its total space. The normal bundle at $(X, \underline{p}, \sO_X)$ is then given by first order deformations $\sO_X$ as a q-stable torsion-free rank $1$ sheaf over fixed $(X, \underline{p})$. Even when looking at $\sO_X$ as a coherent sheaf these deformations are classified by $\Hi^1\big(X, \mc{E}\!nd(\sO_X, \sO_X)\big) = \Hi^1(\sO_X)$. This is enough to check that $N_{\fZ_{g, n}}$ is equal to $\epsilon_n^*\Ri^1\!\pi_*\sO_X$ and the claim follows from a direct use of the projection formula. 
\end{proof}
Previously known expressions of \ref{eq:Z} on partial compactifications of $\cJ_{g, n}$ were based on relation \ref{eq:Theta}, which is only valid over curves of compact type. Correction terms were computed in \cite{GrushZakzerosection} to extend the latter relation to the locus of curves having at most $1$ non-separating node. The advantage of the present formula, from a modular point of view, is that $\mf{L}$ is much easier to deal with than the theta divisor. This will be put in practice in section \ref{sec:AFormulaForPullbacks}.  We identify in the following section the term in \ref{eq:Z} coming from the class of the universal theta divisor $[\Theta]$, trivialized along the zero section.

\subsection{Relating the theta divisor to $\mf{L}$}
Marked points do not play a role in this section. The corresponding compactification $\cJj_g^{d, \mc{P}}$ of the universal jacobian $\cJ_g$ over (unmarked) stable curves is given by semi-stable torsion-free sheaves with respect to a given polarisation $\mc{P}$ over $\cMm_g$, see section \ref{sec:CompactifiedUniversalJacobian} for details. We will only be concerned by the degree $0$ and $g-1$ compactified universal jacobians $\cJj_{g}$ and $\cJj_g^{g-1}$; the former has polarisation $\omega^{-1}\otimes \sO$ while the latter has trivial polarisation $\sO$. The compactified universal jacobian $\cJj_{g}^{g-1}$ comes with a canonical theta divisor $\Theta_{g-1}$ set theoretically given by
\begin{equation}
  \label{eq:defnThetag}
  \Theta_{g-1} = \big\{ (X, L) \mid h^0(L)\geq 1\big\}.
\end{equation}
We have the relation 
\begin{equation}
  \label{eq:ThetaDet}
  -\Theta_{g-1} = d_\pi\big(\mc{L}_{g-1}\big)
\end{equation}
where $d_\pi(\bullet)$ is the determinant in cohomology $\det(\pi_!\bullet)$, see for instance \cite[13]{ArbaCorGriff}. This description can be found in \cite{MR2105707} and \cite{MR2557139}. We question how to pullback $\Theta_{g-1}$ on $\cJj_g$. 

For the degree $0$ jacobian of a single smooth curve $X$, defining a theta divisor is about pulling back $\Theta_{g-1}$ along an isomorphism $ - \otimes \xi : \Jac(X) \simeq \Jac^{g-1}(X)$ for a degree $g-1$ line bundle $\xi$. A line bundle on $\Jac(X)$ is said to be symmetric if it is invariant by pullback along the involution $L \mapsto L^\vee$. Using the set-theoretic description \ref{eq:defnThetag} one can show that the pullback of $\Theta_{g-1}$ along $-\otimes \xi$ is symmetric if and only $\xi$ is a theta characteristic, i.e. if and only if $\xi^{\otimes 2} \simeq \omega$. There is no canonical choice of a theta characteristic for a given curve, there is thus no hope to define a symmetric theta divisor this way on $\cJ_g$. There is, however, a roundabout way to partially tackle this question: the pullback of the rational equivalence class of $2\Theta_{g-1}$ by $-\otimes \xi$ does not depend on the theta characteristic $\xi$. Indeed, two theta characteristics $\xi$ and $\xi'$ differ by a root of the trivial line bundle $\eta = \xi\otimes \xi'^{-1}$. Using the theorem of the square \cite[II. 6.]{MR2514037} we have that 
\begin{equation}
  \label{eq:equivrattheta}
2\Theta_{g-1} \sim 2t_\eta\Theta_{g-1},
\end{equation}
where $t_{\eta}$ is translation by $\eta$ on the Picard group of $\Jac^{g-1}(X)$. Pulling back \ref{eq:equivrattheta} along $-\otimes \xi$ we get the desired relation between the pullbacks along $-\otimes \xi$ and $-\otimes \xi'$ of $2\Theta_{g-1}$.

We should thus be able to pullback $2\Theta_{g-1}$ to $\cJ_g$. The point is that there is no global choice of a theta characteristic on $\cM_{g}$, one has to take in account all such possible choices. Assume for now $g \geq 2$ and let $\smash{\mc{S}_{g}^{2}}$ be the moduli space of smooth genus $g$ spin curves, see for instance \cite{MR1082361}. A point of $\smash{\mc{S}_{g}^2}$ is a couple $(X, \xi)$ where $X$ is a marked curve of genus $g$ and $\xi$ a line bundle of degree $g-1$ together with an isomorphism $\alpha : \xi^{\otimes 2}\simeq \omega_{X}$. The moduli space $\smash{\mc{S}_{g}^2}$ comes with a forgetful map to $\cM_{g}$ only keeping the underlying marked curve, and a universal line bundle $\Xi$ on the universal curve such that $\Xi^{\otimes 2} \simeq \omega$. On the fiber product $\mc{S}_{g}^{2} \times_{\cM_{g}} \cJ_{g}$ we have two maps 
\begin{equation}
  \begin{tikzpicture}[>=latex, baseline=(current  bounding  box.center)]
    \matrix (m) 
    [matrix of math nodes, row sep=2.5em, column sep=2.5em, text
    height=1ex, text depth=0.25ex]  
    { \cJ_g & \mc{S}_{g}^2\times_{\cM_g}\cJ_{g} & \cJ_{g}^{g-1}\\}; 
    \path[->,font=\scriptsize]  
    (m-1-2) edge node[above] {$\beta$} (m-1-1)
    (m-1-2) edge node[above] {$\alpha$} (m-1-3);
  \end{tikzpicture}
\end{equation}
where $\beta$ is the projection on the second factor and $\alpha$ sends $(f : \mc{X} \rightarrow S, \mc{L}, \xi)$ on $(f : \mc{X} \rightarrow S, \mc{L}\otimes \xi)$. By the standard theory of spin curves, the map $\beta$ is finite of degree $2^{2g}$. Write $\langle \bullet, - \rangle$ for the Deligne pairing of two line bundles $\bullet$ and $-$ on the universal curves over $\cJ_{g}^{g-1}$ or $\cJ_g$. Using \ref{eq:ThetaDet} we have that 
\begin{displaymath}
  -2\alpha^*\Theta_{g-1}  = d_\phi(\beta^*\mf{L}\otimes \Xi)^{\otimes 2}.
\end{displaymath}
Using GRR \cite[XII 6. 5.31]{ArbaCorGriff}, we can write
\begin{align*}
  -2\alpha^*\Theta_{g-1}   & = \langle \beta^*\mf{L}\otimes \Xi, \beta^*\mf{L}\otimes \Xi^{-1} \rangle d_\phi(\sO)^{\otimes 2} \\
                          & = \langle \beta^*\mf{L}, \beta^*\mf{L} \rangle \langle \beta^*\mf{L}, \Xi \rangle \langle \beta^*\mf{L} , \Xi^{-1} \rangle \langle \Xi, \Xi^{-1} \rangle d_\phi(\sO)^{\otimes 2}.   
\end{align*}
Since $\langle \bullet, \Xi^{-1} \rangle = \langle \bullet, \Xi \rangle^{-1}$ we get that
\begin{displaymath}
  -2\alpha^*\Theta_{g-1} = \langle \beta^*\mf{L}, \beta^*\mf{L} \rangle \langle \Xi, \Xi \rangle^{-1}d_\phi(\sO)^{\otimes 2}.
\end{displaymath}
One can check by looking at the Chern classes in the group of relative Cartier divisors of $\mc{S}_{g}^2\times_{\cM_g} \cJ_g$ that 
\begin{displaymath}
  c_1\langle \Xi, \Xi^{-1} \rangle = -\phi_*\left(\frac{c_1(\omega)^{2}}{4}\right) = -3\lambda_1 = c_1\left(d_\phi(\sO)^{- \otimes 3}\right).
\end{displaymath}
The middle equality comes from the Mumford relation over $\cM_g$ \cite[XII 7.3]{ArbaCorGriff} and $\lambda_1$ is the first Chern class of the pullback to $\mc{S}_g^2\times_{\cM_g}\cJ_g$ of the Hodge bundle over $\cM_g$. Putting the previous computations together we get that
\begin{equation}
  \label{eq:pullbackthetag}
  -2\alpha^*\Theta_{g-1} = \beta^*\left(\langle \mf{L}, \mf{L} \rangle d_\phi(\sO)^{- 1}\right).
\end{equation}
\begin{rem}
  Computations in the genus $1$ case are similar. One has to be careful though, in order to work with Deligne-Mumford stacks one has to stick to working with $\cM_{1, 1}$. We get the same expression as \ref{eq:pullbackthetag} but over $\cM_{1, 1}$. 
\end{rem}
\begin{rem}
  \label{rem:extpairing}
  We extend the Deligne pairing to the case of $\mf{L}$ and $\mf{L}^{\vee}$ over $\cJj_{g}$ using the relation
  \begin{displaymath}
    \langle L, M \rangle = d_\phi(L\otimes M)d_\phi(L)^{-1}d_\phi(M)^{-1}d_\phi(\sO)
  \end{displaymath}
  for $L$ and $M$ invertible sheaves or in $\{\mf{L}^{\otimes m} , \mf{L}^{\vee \otimes m} \mid m \in \N \}$. All standard computational properties including the Serre duality extend to this case. A way to check this fact is by recalling that $\cJj_{g}^d$ is isomorphic to the compactification by Caporaso of the universal jacobian using quasi-stable curves. The pullback of $\mf{L}$ to the latter compactification is a well-behaved invertible sheaf. See \cite{EstevesPaciniSemistablemodifications} for details. One can thus check the standard properties extend as claimed by pullback on Caporaso's compactification using semi-stable line bundles on quasi-stable curves.
\end{rem}
\begin{defn}
  A squared theta divisor over $\cJj_g$ is a divisor given on each geometric fiber over $\cM_g^{ct}$ as the pullback of $2\Theta_{g-1}$ along a theta characteristic. It is said to be 
  \begin{itemize}
    \item
      trivialized along the zero section if its pullback along the zero section $\cMm_{g} \rightarrow \cJj_{g}$ is trivial
    \item 
      symmetric if it is invariant under the involution of the Picard group of $\cJj_{g \mid 1}$ sending $\mf{L}$ on $\mf{L}^\vee$.
\end{itemize}
In case of genus $1$ the definition makes sense replacing $\cJj_g$ and $\cMm_g$ respectively by $\cJj_{1, 1}$ and $\cMm_{1, 1}$. 
\end{defn}
\begin{prop}
  There is a unique symmetric squared theta divisor $2\Theta$ over $\cJj_g$ which is trivialized along the zero section. It is then given by the formula
  \begin{equation}
    \label{eq:ThetaLl}
    2\Theta = \langle \mf{L}, \mf{L} \rangle^{-1}.
  \end{equation}
\end{prop}
\begin{proof}
  The fact the right-hand side of \ref{eq:ThetaLl} is symmetric comes from the standard properties of the Deligne pairing in the case of torsion free sheaves as mentioned in remark \ref{rem:extpairing}.

  Over $\cM_g$ the pullback of $2\Theta_{g-1}$ is given by formula \ref{eq:pullbackthetag}. The line bundle $\langle \mf{L}, \mf{L}\rangle^{-1} d_\phi(\sO)$ is thus a squared theta divisor over $\cJ_g$. Its pullback along the zero section is given by $d_\phi(\sO)$. It is enough to get the desired formula over $\cJ_g$. To extend it on $\cJ_{g}^{ct}$ on only needs to check formula \ref{eq:ThetaLl} on the locus of curves having two components. Let $\imath_{h} : \cJ_{h, x} \times \cJ_{g-h, y} \rightarrow \cJ_{g}^{ct}$ be the gluing maps identifying $x$ and $y$ and write $p_i$ for the first and second projections from $\cJ_{h, x}\times \cJ_{g-h, y}$ respectively on $\cJ_{h, x}$ and $\cJ_{g-h, y}$. The universal curve over $\cJ_{h, x}\times\cJ_{g-h, y}$ is the disjoint union
  \begin{equation}
    \label{eq:univcurveboundary}
    \left(\cJ_{h, x \mid 1}\times \cJ_{g-h, y}\right) \bigsqcup  \left(\cJ_{h, x} \times \cJ_{g-h, y \mid 1}\right)
  \end{equation}
with the obvious maps $\phi$ and $\imath_{h \mid 1}$ respectively to $\cJ_{h, x}\times \cJ_{g-h, y}$ and $\cJ_{g}^{ct}$. Write $p_h$ and $p_{g-h}$ for the maps whose domain is \ref{eq:univcurveboundary} and only keeping $\cJ_{h, x \mid 1}$ and $\cJ_{g-h, y \mid 1}$. We have that $\phi \circ p_h = p_1$ and $\phi\circ p_{g-h} = p_2$. Using the simple relation  $\imath_{h \mid 1}^*\mf{L} = p_h^*\mf{L} + p_{g-h}^*\mf{L}$ on the corresponding universal curves we get that
\begin{align}
  \label{eq:ThetaCompacttype}
  \imath_{h}^*\langle \mf{L}, \mf{L} \rangle & = \langle p_h^*\mf{L}, p_h^*\mf{L} \rangle \langle p_h^*\mf{L}, p_{g-h}^*\mf{L}\rangle^{\otimes 2}\langle p_{g-h}^*\mf{L}, p_{g-h}^*\mf{L} \rangle \\
  & = \langle p_h^*\mf{L}, p_h^*\mf{L} \rangle \langle p_{g-h}^*\mf{L}, p_{g-h}^*\mf{L} \rangle. \nonumber
\end{align}
The middle term vanishes because $p_h^*\mf{L}$ and $p_{g-h}^*\mf{L}$ have disjoint supports. Each term of the right-hand side of \ref{eq:ThetaCompacttype} respectively correspond to the pullbacks $-2p_1^*\Theta$ and $-2p_2^*\Theta$ because of the previous work on the smooth case. Now, the equality
\begin{displaymath}
  \imath_{h}^*\langle \mf{L}, \mf{L} \rangle = -2\Theta
\end{displaymath}
is just about rewriting the fact that the theta divisor of a decomposable jacobian is the union of pullbacks along the projection to each factor.

Uniqueness is a consequence of the work of Melo and Viviani \cite{MeloVivianiPicardGroup} on the Picard group of $\cJj_g$. Using our notation theorem B of \textit{loc. cit.} states that the Picard group of $\cJj_g$ is freely generated by $\langle \mf{L}, \mf{L} \rangle d_\phi(\sO)^{\otimes 2}$, $d_\phi(\omega)$ and line bundles obtained by pullback from $\cMm_g$. Thus, the only extension of $2\Theta$ to $\cJj_g$ that is trivialized along the zero section is the one given by formula \ref{eq:ThetaLl}. 
\end{proof}
\begin{cor}
  Let $[\Theta]$ be half the class of the symmetric squared theta divisor trivialized along the zero section in the (rational) Chow ring of $\cJj_{g}$. We then have 
  \begin{equation}
    \label{eq:ThetaL}
    [\Theta] = - \phi_*\left(\frac{c_1(\mf{L})^2}{2}\right).
  \end{equation}
\end{cor}
\begin{proof}
This is a direct computation of first Chern classes of both sides of \ref{eq:ThetaLl}.
\end{proof}
\begin{rem}
  In the case of marked underlying curves, the theta divisor is obtained by pullback along the forgetful morphism $\cJj_{g, n} \rightarrow \cJj_g$ in the case of genus $g \geq 2$ and by pullback along $\cJj_{1, n} \rightarrow \cJj_{1, 1}$ for the genus $1$ case. Formula \ref{eq:ThetaLl} does still hold in these cases. 
\end{rem}
Putting back relation \ref{eq:ThetaL} into the expression of the zero cycle \ref{eq:Z} one can rewrite \ref{eq:Z} in the form
  \begin{align*}
  [\fZ_{g, n}]   & =  \bigg\{ \exp \left( -\phi_* \left( \frac{c_1(\mf{L})^2}{2} \right) \right)\times\exp \left(\frac{c_1(\mf{L})c_1(\omega_\phi)}{2} - \frac{B_2}{2}(c_1(\omega_\phi)^{2} + \Delta)\right) \\
 & \qquad  \times\exp\big(\sum_{s\geq 2}(-1)^s(s-1)! \Big( \phi_* (\ch(\mf{L})\Td^\vee(\Omega_{\phi})) \Big)_s \bigg\}_g 
\end{align*}
where $\omega_\phi$ is the dualizing sheaf of $\phi$, $\Delta$ the singular locus in the fibers of $\phi$. This gives 
\begin{equation}
  \label{eq:ZTheta}
    [\fZ_{g, n}]  =  \sum_{a+b = g} \frac{[\Theta]^a}{a!} P_b 
\end{equation}
where $P_b$ is a polynomial in pushforward of classes on $\cJj_{g, n \mid 1}$ lying in the degree $b$ part of the Chow ring. It is not clear how to write $P_b$ as a polynomial in an elegant minimal set of classes on $\cJj_{g, n}$, the main reason being that not much is known about the intersection theory of $\cJj_{g, n}$. Relation \ref{eq:Theta} (pulled back to the universal jacobian) does however imply that on the locus $\cM_{g, n}^{ct}$ of curves of compact type we have
\begin{equation}
  \label{eq:VanishingTheta}
  \sum_{a + b = g} \frac{[\Theta]^{a}}{a!}P_b = \frac{[\Theta]^g}{g!}.
\end{equation}

\section{Partial extension of Abel-Jacobi maps}
\label{sec:PartialExtensionofAbel-JacobiMaps}

Let $\underline{\tau} = (\tau_1, \ldots, \tau_n)$ be a non-zero $n$-tuple of integer entries having zero sum. The forgetful map $\epsilon_n : \cJ_{g, n} \rightarrow \cM_{g, n}$ sending an $S$-point of $\cJ_{g, n}$ on the underlying marked stable curve has a section $\aj$--known as the Abel-Jacobi map--sending an $S$-point $(f : \mc{X} \rightarrow S, \underline{D})$ on $(f : \mc{X} \rightarrow S, \underline{D}, \sO_{\mc{X}}(\sum_i \tau_iD_i))$. More generally given an integer $k \in \Z$ and an $n$-tuple of integers $(\tau_1, \ldots, \tau_n)$ having sum equal to $k(2g-2)$ one can define the section of $\epsilon_n$
\begin{displaymath}
  \aj_k(f : \mc{X} \rightarrow S, \underline{D}) = \Big(f : \mc{X} \rightarrow S, \underline{D}, \sO_{\mc{X}}(\sum_i\tau_iD_i)\otimes \omega^{-k}\Big).
\end{displaymath}
These Abel-Jacobi maps can be extended to partial compactifications of $\cM_{g, n}$. The extension to curves of compact type is standard and the one to curves with at most $1$ non-separating node was described in \cite{GrushZakzerosection}. These extensions appear to give partial answers to the Eliashberg problem: express the pullback of $[\fZ_{g, n}]$ along extensions of $\aj_0$ to partial compactifications of $\cM_{g, n}$. 

We will extend the $\aj_k$s yet further. We start first by noticing there is an open substack of $\cMm_{g, n}$ where $\aj_k$ makes sense without any change. We then review the extension to the case of curves of treelike type. The indeterminacy locus of $\aj_k$s is contained in the complement of the union of both previous loci.

\subsection{The locus of balanced curves}

Given a family of curves $f : \mc{X} \rightarrow S$ and a line bundle $\mc{L}$ over $\mc{X}$ the set of points $s \in S$ where $\mc{L}_s$ is q-stable in the sense of definition \ref{defn:semistability} is open in $S$, see for instance \cite[1.4]{EstevesCompactifyingJacobian}. 
\begin{defn}
  A curve $(X, \underline{p})$ in $\cMm_{g, n}$ is said to be $(\underline{\tau}, k)$-balanced if for each subcurve 
\begin{equation}
\label{eq:balanced}
Z \subset X, \quad \sum_{p_i \in Z} \tau_i \geq k\deg(\omega_{X \mid Z}) -\frac{\kappa_Z}{2}
\end{equation}
with strict inequality whenever $p_1 \in Z$. 
\end{defn}
This condition is the one for the q-stability of $\sO_X(\sum_{i=1}^n \tau_ip_i)\otimes \omega_X^{-k}$. It defines an open substack $\cB_{g, n}^{k, \underline{\tau}}$ of $\cMm_{g, n}$. It is clear that the section $\aj_k$ extends to $\cB_{g, n}^{k, \underline{\tau}}$. The extension is denoted by $\aj_k$ as well.

\subsection{Further extensions and twisters}

Over the complement of $\cB_{g, n}^{k, \underline{\tau}}$ the line bundle $\sO(\sum_i\tau_iD_i)\otimes \omega^{-k}$ is no more q-stable; the section $\aj_k$ does not make sense anymore as it is. Notice however that any extension of $\aj_k$ to a point in $\cMm_{g, n}-\cB_{g, n}^{k, \underline{\tau}}$ gives a point in the closure $\overline{\aj_k\cB_{g,n}^{k, \underline{\tau}}}$. Such a point in the image appears as the special fiber of a family $(f : \mc{X} \rightarrow \Spec(R), \underline{D}, \mc{L})$ where $R$ is a DVR and $\mc{L}$ is a q-stable torsion-free sheaf whose restriction to the generic fiber $\mc{X}_{\eta}$ is given by $\sO_{\mc{X}_{\eta}}(\sum_i\tau_iD_{i, \eta})\otimes \omega_{\mc{X}_\eta}^{-k}$. Properness of $\cJj_{g, n}$ says that $\mc{L}$ is unique up to tensor product with an invertible sheaf coming from the base. We shortly recall facts about twisters to make clear how to extend $\aj_k$ on a bigger substack in $\cMm_{g, n}$. Basic facts about twisters can be found in \cite{CaporasoNerontype}. 

Let $(f : \mc{X} \rightarrow T, \underline{D})$ be a regular model over a DVR $R$ having smooth generic fiber $\mc{X}_\eta$ and central fiber $X$. Write $X_v$ for the irreducible component of $X$ corresponding to the vertex $v \in V$ of the dual graph of $X$. Over the generic smooth fiber the line $\mc{L}_\eta = \sO_{\mc{X}_\eta}(\sum_i\tau_iD_i)\otimes \omega^{-k}$ is q-stable. The valuative criterion of properness for $\cJj_{g, n}$ implies $\mc{L}_{\eta}$ extends uniquely (up to tensor product by a pullback from the base) into a q-stable line bundle $\mc{L}$ over $\mc{X}$. We are going to build such limits for treelike curves by tensoring the line bundle $\sO_{\mc{X}}(\sum_i\tau_iD_i)\otimes \omega^{-k}$ with suitable twisters over $\mc{X}$. A treelike curve is a stable curve such that each edge joining two different vertices of its dual graph is separating, i.e. taking out that edge gives a graph with an extra connected component. 

Recall that a twister is a line bundle over $X$ of the form $\sO_{\mc{X}}(\sum_v\gamma_vX_v)\otimes \sO_X$. The set of twisters ${\rm{Tw}}_f$ acts faithfully on $\sO_{\mc{X}}(\sum_i\tau_iD_i)\otimes \omega^{-k}$ by tensor product. 
\begin{lem}
  If $X$ is treelike then any q-stable line bundle $L$ on $X$ has $\underline{0}$ multidegree. 
\end{lem}
\begin{proof}
  If $X$ has only one irreducible component there is nothing to prove. Assume $X$ has at least $2$ components. Let $v$ be a leaf of the dual graph $\Gamma_X$ and $X_v$ the corresponding irreducible component. The q-stability condition states 
\begin{displaymath}
  -\frac{1}{2} \leq \deg_{X_v}\!L \leq \frac{1}{2}
\end{displaymath}
and thus $L$ has degree $0$ on $X_v$. Write $X'$ for $X_v^c = \overline{X - X_v}$. Taking any subcurve $Z = Y\cup X_v$ we have the inequality
\begin{displaymath}
  -\frac{\kappa_{Z}}{2} \leq \deg_{Z}\!L \leq \frac{\kappa_Z}{2}.
\end{displaymath}
The middle term is in fact equal to $\deg_Y\!L$ and $\kappa_Z = \kappa_{Y\mid X'}$, the number of intersection points of $Y$ with $Y^c$ in $X'$. We thus get the inequality 
\begin{displaymath}
  -\frac{\kappa_{Y}}{2} \leq \deg_{Y}\!L \leq \frac{\kappa_Y}{2}
\end{displaymath}
in $X'$. Since this condition is harder than the one for $Y$ in $X$, it is enough to check the multidegree of $L$ on $X'$ is $0$. Since $X'$ has less components than $X$ the result is obtained by induction. 
\end{proof}
\begin{rem}
  In general a q-stable line bundle is not of multidegree $\underline{0}$. The simplest example comes from a banana curve.
\end{rem}
\begin{lem}
  \label{lem:twister}
Assume $X$ is treelike. There is then a unique q-stable line bundle in the orbit of $\sO_{\mc{X}}(\sum_i\tau_iD_i)\otimes \omega^{-k}$ under the action of $\mr{Tw}_{f}$.  
\end{lem}
\begin{proof}
  The point is to prove that, given a line bundle $\mc{L}$ on $\mc{X}$, there is a twister sending $\mc{L}_{\mid X}$ to a multidegree $\underline{0}$ line bundle on $X$. A short answer would be to recall that the class group of a treelike graph is trivial, see for instance \cite{CaporasoNerontype}. We shall however go more into details in order to keep track of how to get the desired twister.

Let $\Gamma$ be the dual graph of $X$ and $V$ the set of vertices of $\Gamma$. The irreducible component of $X$ corresponding to a vertex $v \in V$ is denoted $X_v$. Let $L$ be the restriction $\mc{L}_{\mid X}$ and write $(m_v^0)_{v \in V}$ for the multidegree of $L$ on $X=X_0$. Since $\Gamma$ is treelike there is a filtration covering $\Gamma$
\begin{displaymath}
  \Gamma = \Gamma_0 \supsetneq \Gamma_1 \supsetneq \cdots \supsetneq \Gamma_{|V|-1} \supsetneq \emptyset
\end{displaymath}
built up by choosing leaves $v_r$ of the loopless graph obtained out of $\Gamma_r$ by contracting possible loops and satisfying
\begin{displaymath}
  \Gamma_{r+1} = \Gamma_r - v_r.
\end{displaymath}
The right hand side of the previous equality is short for the subgraph of $\Gamma_r$ generated by all vertices except $v_r$. It does contain all edges that are not half-edges of $v_r$ or in geometric edges incident to $v_r$. For each $v_r$ let $B_{v_r}$ be the branch of $\Gamma$ out of $v_r$ ; removing the unique edge linking $v_r$ to $\Gamma_{r+1}$, $B_{v_r}$ is the component containing $v_r$. Define on $\mc{X}$ the following sequence $(\mc{L}_r)_{0 \leq r \leq |V|-1}$ of line bundles
\begin{enumerate}
\item $\mc{L}_0 = \mc{L}$
\item let $L_r = \mc{L}_{r \mid X}$ and write $(m_{v}^r)_{v \in V}$ for the multidegree of $L_r$, then
  \begin{equation}
    \label{eq:Li}
    \forall r \geq 0, \quad \mc{L}_{r + 1} = \mc{L}_r\Big(m_{v_{r}}^{r} \sum_{w \in B_{v_r}} X_{w}\Big).
  \end{equation}
\end{enumerate}
Each $L_r$ is obtained out of $L = L_0$ by tensor product with a given twister. Let $F_r$ be the set of vertices of $\Gamma$ not in $\Gamma_r$. This is a strictly increasing finite collection of sets. 
\begin{Claim}
  For all $j \geq r$, $L_j$ has zero degree on each component $X_v$ for $v \in F_r$. 
\end{Claim}
\noindent Notice first that each $v \in F_r$ is of the form $v_\ell$ for some $\ell \leq r$. It is thus enough to show that for all $j \geq r$, $L_j$ has degree $0$ on $v_r$. The claim is clear if $r=0$, assume $r \geq 1$. The degree of $L_{r}$ on $v_{r}$ is given by 
\begin{equation}
  \label{eq:Lr}
  \deg\big(L_{r \mid X_{v_{r}}}\big) = \underbrace{\deg\big(L_{r-1 \mid X_{v_{r-1}}}\big)}_{(I)} + \underbrace{\deg\Big(\sO_{\mc{X}}\big(m_{v_{r-1}}^{r-1}\sum_{w \in B_{v_{r-1}}} X_w\big)_{\mid X_{v_{r-1}}}\Big)}_{(II)}.
\end{equation}
By definition $(I)$ is equal to $m_{v_{r-1}}^{r-1}$. The term $(II)$ is equal to 
\begin{displaymath}
  m_{v_{r-1}}^{r-1}\big(\val(v_{r}) -1 \big) - m_{v_{r-1}}^{r-1}\val(v_{r}) = - m_{v_{r-1}}^{r-1}.
\end{displaymath}
The first term of the left hand side is the contribution of neighboring vertices of $v_{r-1}$ while the second one comes from the self-intersection of $X_{v_{r-1}}$. The total sum $(I) + (II)$ is thus $0$ as claimed. Now, for any $j > r$ we have that
\begin{displaymath}
  \deg\big(L_{j \mid X_{v_{r}}}\big) = \underbrace{\deg\big(L_{r \mid X_{v_r}}\big)}_{= 0} + \sum_{r < \ell \leq j} \deg\Big(\sO_{\mc{X}}\big(m_{v_r}^r\sum_{w \in B_{v_r}} X_w\big)_{\mid X_{v_r}}\Big).
\end{displaymath}
If $v_r \notin B_{v_\ell}$ for some $\ell$ then the corresponding term in the right hand sum is clearly $0$. Otherwise we have that
\begin{displaymath}
  \deg\Big(\sO_{\mc{X}}\big(m_{v_r}^r \sum_{w \in B_\ell}X_{w}\big)_{\mid X_{v_r}}\Big) = m_{v_r}^r\val(v_r) - m_{v_r}^r\val(v_r) = 0
\end{displaymath}
where the first right hand term comes from contributions of neighbors of $v_r$, the second from the self intersection of $X_{v_r}$. 

Our claim is proved. We have shown that we have an increasing filtration by $F_r$s that covers $V$ and on which $\mc{L}_{r \mid X}$ has $\underline{0}$ multidegree. There is therefore a line bundle $\mc{L}_\infty$ whose restriction to $X$ has multidegree $\underline{0}$. 
\end{proof}

\begin{rem}
  Each component of the multidegree $(m_v^r)_{v \in V}$ of $L_r$ can be written explicitly in terms of the combinatorics of $\Gamma$. We have already seen during the previous proof that given $v \in F_r$
  \begin{displaymath}
    \forall j \geq r,\quad m_{v}^j = \deg\big(L_{j \mid X_{v}}\big) = 0.
  \end{displaymath}
  In general we have the relation
  \begin{displaymath}
    \deg\big(L_{j\mid X_v}\big) = \deg\big(L_{j-1 \mid X_v}\big) + \deg\Big(\sO_{\mc{X}}\big(m_{v_{j-1}}^{j-1}\sum_{w \in B_{v_j}} X_w\big)_{\mid X_{v}}\Big)
  \end{displaymath}
  which can also be written 
  \begin{displaymath}
    m_{v}^j = m_v^{j-1} + \deg\Big(\sO_{\mc{X}}\big(m_{v_{j-1}}^{j-1}\sum_{w \in B_{v_j}} X_w\big)_{\mid X_{v}}\Big).
  \end{displaymath}
  When $ j < r$, the vertex $v$ doesn't appear in $B_{v_{j-1}}$. The second term of the previous right hand side is $0$ unless $v_j$ is adjacent to $v$, in which case it is $1$. Let $\chi_{v_{j-1}}$ be the function whose value is $0$ on a vertex $v$ if $v_{j-1}$ is adjacent to $v$ and $1$ otherwise. We thus have that
  \begin{displaymath}
    m_v^j = m_v^{j-1} + \chi_{v_{j-1}}(v)m_{v_{j-1}}^{j-1}.
  \end{displaymath}
  An inductive argument thus gives
  \begin{displaymath}
    m_v^j = \sum_{\ell < j} \chi_{v_{\ell}}(v)m_{v_\ell}^{\ell} = \sum_{w \in B_{v}} m_w.
  \end{displaymath}
  In particular one can write $\mc{L}_{r+1}$ as 
  \begin{equation}
    \label{eq:twisterr}
    \mc{L}_{r+1} = \mc{L}_r\Big(\big(\sum_{w \in B_{v_r}} m_w\big)\sum_{w \in B_{v_r}} X_w\Big).
  \end{equation}
\end{rem}

\subsection{Defining a partial extension}

Let $\cMm_{g, n}^{k, \underline{\tau}}$ be the open substack union of the locus of curves of treelike type and the one of $(\underline{\tau}, k)$-balanced ones $\cB_{g, n}^{k, \underline{\tau}}$. This open substack can be described as the complement in $\cMm_{g, n}$ of the following closed locus: let $\cI_{g, n}^{k, \underline{\tau}}$ be the closure in $\cMm_{g, n}$ of curves having topological type a two vertex loopless graph, which is not $(\underline{\tau}, k)$-balanced; then $\cMm_{g, n}^{k, \underline{\tau}} = \cMm_{g, n} - \cI_{g, n}^{k, \underline{\tau}}$. Write $\cJj_{g, n}^{k, \underline{\tau}}$ for the substack of $\cJj_{g, n}$ obtained by pullback along $\cMm_{g, n}^{k, \underline{\tau}} \subset \cMm_{g, n}$. 

Recall that in $\cMm_{g, n+1}$ the divisor $\delta_{h, A}$ is the closure of stable marked curves of topological type a graph having only one edge and two vertices, one of which is of genus $h$ and whose set of legs (marked points) is $A$. The divisor $D_i$ corresponding to the image of the sections of $\pi : \cMm_{g, n+1} \rightarrow \cMm_{g, n}$ is simply $\delta_{0, \{n+1, i\}}$. 
\begin{prop}
  \label{prop:qk}
  Let $\mc{L}(\underline{\tau}, k)$ be the line bundle on $\cMm_{g, n+1}$ given by
  \begin{equation}
    \label{eq:qk}
    \mc{L}(\underline{\tau}, k) = \sO\Big(\sum_{i=1}^n \tau_iD_i\Big)\otimes \omega^{-k}\otimes\sO\Big(\sum_{\substack{0 \leq h \leq \lfloor g/2 \rfloor \\ A \subset \{1, \ldots, n\} \\ 2 \leq |A| + h \leq g+ n-2 }} \big[k(1-2h)+\sum_{i \in A} \tau_i\big]\delta_{h, A\cup\{n+1\}}\Big).
  \end{equation}
 Then $\mc{L}(\underline{\tau}, k)$ defines a section $\bar{\aj}_k : \cMm_{g, n}^{k, \underline{\tau}} \rightarrow \cJj_{g, n}^{k, \underline{\tau}}$ extending $\aj_k$.
\end{prop}
\begin{proof}
  We need to check that along fibers of $\pi$ the line bundle $\mc{L}(\underline{\tau}, k)$ is q-stable, i.e. has multidegree $\underline{0}$. Let $(X, \underline{p})$ be a treelike stable $n$-marked curve over a point. Write $\Gamma$ for the dual graph of $X$ and $V$ for its set of vertices. Let $X_v$ be the irreducible component of $X$ corresponding to $v \in V$ and $L_v$ the set of marked points of $X_v$. 

  Choose a regular smoothing $f : \mc{X} \rightarrow S$ of $X$ and let $D_i$ be the corresponding marked point of $f$. If $X$ is a point in $\delta_{h, A} \subset \cMm_{g, n}$ write $\Lambda_{h, A}$ for the subgraph of $\Gamma$ that comes from a specialization of the genus $h$ $A$-marked vertex of the generic point of $\delta_{h, A}$. The pull back of $\delta_{h, A\cup\{n+1\}}$ over $f$ gives the vertical divisor
  \begin{displaymath}
    \Xi_{h, A} = \sum_{v \in \Lambda_{h, A}} X_v.
  \end{displaymath}
The $\delta_{h,A}$ contribution in \ref{eq:qk} is given by
\begin{equation}
  \label{eq:1}
  \big[k(1-2h) + \sum_{i \in A} \tau_i\big]\Xi_{h, A}.
\end{equation}
Now, writing $\hat{g}(v) = g(v) + b_1(v)$ where $g(v)$ is the genus of $v$ (its weight in the dual graph) and $b_1(v)$ is the number of loops $v$ has, we have that
\begin{equation}
\label{eq:2}
  k(1-2h) + \sum_{i \in A} \tau_i = \sum_{v \in \Lambda_{h, A}} \big(k(2-2\hat{g}(v)+\kappa_v) +  \sum_{i \in A\cap L_v} \tau_i\big).
\end{equation}
When $X$ is $(\underline{\tau}, k)$-balanced \ref{eq:2} is $0$ by assumption. Thus $\mc{L}(\underline{\tau}, k)$ is the desired line bundle in the $(\underline{\tau}, k)$-balanced case. When $X$ is treelike notice that each subgraph $\Lambda_{h, A}$ is a branch of $\Gamma$. The vertex of that branch corresponds to the nodal point linking the subcurve of $X$ in $\cMm_{h, A}$ to the one in $\cMm_{g-h, \complement A}$. One can order $\delta_{h, A\cup\{n+1\}}$s in such a way that complements of $\Lambda_{h, A}$s give a decreasing finite cover of $\Gamma$ as in the setting we started with in the proof of lemma \ref{lem:twister}. From \ref{eq:2} and \ref{eq:1} the contribution of $\delta_{h, A\cup\{n+1\}}$ is given by
\begin{displaymath}
  \sO_{\mc{X}}\Big(\Big[\sum_{v \in \Lambda_{h, A}} \big(k(2-2\hat{g}(v)+\kappa_v) +  \sum_{i \in A\cap L_v} \tau_i\big)\Big]\sum_{v \in \Lambda_{h, A}} X_v\Big)
\end{displaymath}
By \ref{eq:twisterr} this is precisely the contribution we need to build up a q-stable line bundle out of $\sO_{\mc{X}}(\sum_{i=1}^n \tau_iD_i)\otimes \omega^{-k}$ by tensoring with twisters.
\end{proof}

\subsection{Obstruction to extend Abel-Jacobi maps yet further}
    Banana curves are obtained by identifying two distinct points on a smooth curves with two others from another curve. Given such a curve $(X, \underline{p})$, there are only two possible q-stable mutlidegrees on $X$ given by $(0, 0)$ and $(1, -1)$, the first coordinate corresponding to the component marked by $p_1$. Given a regular one parameter smoothing $(f : \mc{X} \rightarrow S, \underline{D})$ there is only one twister getting $\sO_{\mc{X}}(\sum_i\tau_iD_i)$ to a q-stable line bundle having either degree $(0,0)$ or $(1, -1)$. The expected multidegree is $(0,0)$ if the multidegree of the line bundle $\sO_X(\sum_i\tau_ip_i)$ is even and $(1,-1)$ if it is odd. Despite this pleasant feature, this construction cannot be globalized to the universal family. The reason is that the corresponding twister would have to be trivial over the locus of irreducible curves and thus only have support over the one of banana curves. This locus is of codimension $2$ in $\cMm_{g, n+1}$ and cannot correspond to a divisor on $\cMm_{g, n+1}$. This strongly suggests that one needs to blow up $\cMm_{g, n}$ along the locus of none $(\underline{\tau}, k)$-stable banana curves to be able to extend $\bar{\aj}_k$ yet further.

\section{Pullbacks along Abel-Jacobi maps}
\label{sec:AFormulaForPullbacks}

\subsection{Pulling back the zero section}
Our goal is to compute the pullbacks of the class of the zero section $[\fZ_{g, n}]$ in $\cJj_{g, n}^{k, \underline{\tau}}$ along $\bar{\aj}_k$s. Recall that
\begin{equation}
  \tag{\ref{eq:Z}}
[\fZ_{g, n}] = \bigg\{\exp\Big(\sum_{s\geq 1} (-1)^s(s-1)!\big\{\phi_*(\ch(\mf{L})\Td^\vee(\Omega_\phi))\big\}_s\Big)\bigg\}_g  
\end{equation}
where $\{\bullet\}_\ell$ is the degree $\ell$ part of $\bullet$. One can picture the involved data by looking at the commutative diagram
\begin{equation}
\label{diag:cartesianneq}
\begin{tikzpicture}[>=latex, baseline=(current  bounding  box.center)]
    \matrix (m) 
    [matrix of math nodes, row sep=2.5em, column sep=2.5em, text
    height=1ex, text depth=0.25ex]  
    { \mf{L} & \cJj_{g, n\mid 1}^{k, \underline{\tau}} & \cMm_{g,n+1}^{k, \underline{\tau}}  & \mc{L}(\underline{\tau}, k)  \\
      & \cJj_{g, n}^{k, \underline{\tau}} & \cMm_{g, n}^{k, \underline{\tau}} & \\}; 
    \path[-, font=\scriptsize]
    (m-1-1) edge (m-1-2)
    (m-1-3) edge (m-1-4);
    \path[->,font=\scriptsize]  
    (m-1-2) edge node[above] {$\epsilon_{n\mid 1}$} (m-1-3)
    (m-1-2) edge node[left] {$\phi$} (m-2-2)
    (m-1-3) edge node[right] {$\pi$} (m-2-3)
    (m-2-2) edge node[above] {$\epsilon_n$} (m-2-3)
    (m-2-3) edge [bend left=30] node[below] {$\bar{\aj}_k$} (m-2-2);
  \end{tikzpicture}
\end{equation}
Notice that we have a tautological lift $\hat{\aj}_k : \cMm_{g, n+1}^{k, \underline{\tau}} \rightarrow \cJj_{g, n\mid 1}^{k, \underline{\tau}}$ simply given by sending an $S$-section $(f: \mc{X} \rightarrow S, \underline{D}, D_{n+1})$ on $(f : \mc{X} \rightarrow S, \underline{D}, D_{n+1}, \mc{L}(\underline{\tau},k))$. The corresponding $S$-object in $\cJj_{g, n}^{k, \underline{\tau}}$ is obtained by simply forgetting the extra section $D_{n+1}$. We can thus complete diagram \ref{diag:cartesianneq} to get 
\begin{equation}
  \label{eq:commdiagC}
\begin{tikzpicture}[>=latex, baseline=(current  bounding  box.center)]
    \matrix (m) 
    [matrix of math nodes, row sep=2.5em, column sep=2.5em, text
    height=1ex, text depth=0.25ex]  
    { \mf{L} & \cJj_{g, n\mid 1}^{k, \underline{\tau}} & \cMm_{g,n+1}^{k, \underline{\tau}}  & \mc{L}(\underline{\tau}, k)  \\
      & \cJj_{g, n}^{k, \underline{\tau}} & \cMm_{g, n}^{k, \underline{\tau}} & \\}; 
    \path[-, font=\scriptsize]
    (m-1-1) edge (m-1-2)
    (m-1-3) edge (m-1-4);
    \path[->, red, font=\scriptsize]
    (m-1-2) edge node[black,left] {$\phi$} (m-2-2)
    (m-1-3) edge node[black, right] {$\pi$} (m-2-3)
    (m-2-3) edge [bend left=30] node[black, below] {$\bar{\aj}_k$} (m-2-2)
    (m-1-3) edge [bend right=30] node[black, above] {$\hat{\aj}_j$} (m-1-2);
    \path[->,font=\scriptsize]  
    (m-1-2) edge node[below] {$\epsilon_{n\mid 1}$} (m-1-3)
    (m-2-2) edge node[above] {$\epsilon_n$} (m-2-3);    
  \end{tikzpicture}
\end{equation}
\begin{lem}
  On the level of Chow groups we have the relation $\bar{\aj}_k^*\phi_* = \pi_*\hat{\aj}_k^*$.
\end{lem}
\begin{proof}
  This is a standard argument. It is straightforward to check that the red part of diagram \ref{eq:commdiagC} is in fact cartesian. Now given an element $\alpha$ in the Chow group of $\cJj_{g, n \mid 1}^{k, \underline{\tau}}$. We have that
\begin{displaymath}
  \bar{\aj}_k^*\phi_*(\alpha) = \epsilon_{n *}\big(\phi_*(\alpha)\cap [\fZ_{g, n}]\big) = \epsilon_{n *}\phi_*\big(\alpha \cap \phi^*[\fZ_{g, n}]\big)
\end{displaymath}
since $\phi$ is flat we get that
\begin{displaymath}
  \bar{\aj}_k^*\phi_*(\alpha) = \pi_*\epsilon_{n \mid 1 *}\big(\alpha\cap [\phi^{-1}\fZ_{g, n}]\big)
\end{displaymath}
which, because the red part of diagram \ref{eq:commdiagC} is cartesian, gives
\begin{displaymath}
  \bar{\aj}_k^*\phi_*(\alpha) = \pi_*\hat{\aj}_k^*(\alpha).
\end{displaymath}
\end{proof}
\begin{thm}
The pullbacks $\bar{\aj}_k^*[\fZ_{g, n}]$ are given over $\cMm_{g, n}^{k, \underline{\tau}}$ by the relation
\begin{equation}
\label{eq:finale}
\bar{\aj}_k^*[\fZ_{g, n}] = \bigg\{\exp\Big(\sum_{s\geq 1} (-1)^s(s-1)!\big\{\pi_*\big(\ch(\mc{L}(\underline{\tau}, k)\big)\Td^\vee(\Omega_\pi))\big\}_s\Big)\bigg\}_g
\end{equation}
where $\{\bullet\}_\ell$ is the degree $\ell$ part of $\bullet$. 
\end{thm}
\begin{proof}
  This is a direct consequence of the previous lemma including the fact the red part of \ref{eq:commdiagC} is cartesian.
\end{proof}
\begin{rem}
Formula \ref{eq:finale} is an explicit polynomial in the tautological classes of $\cMm_{g, n}$. Unfortunately, it is not clear how to get back Hain's formula or its extension in \cite{GrushZakzerosection} by directly manipulating \ref{eq:finale}. 
\end{rem}

\subsection{Pullback of the theta divisor to $\cMm_{g, n}$}
We follow notation of \cite[17]{ArbaCorGriff}. Write $\tilde{K}$ for the Chern class of $\omega$ and $\psi_i$ for the Chern class of the cotangent bundle along the marked section $D_i$. The $\psi_i$ classes can be described as pushforwards under $\pi$ by the relation $\psi_i = - \pi_*(D_i^2)$. Using relation \ref{eq:ThetaL} we have that 
\begin{equation}
  \label{eq:PullbackTheta1}
  \bar{\aj}_k^*[\Theta] = -\pi_*\left(\frac{c_1\big(\mc{L}(\underline{\tau}, k)\big)^2}{2}\right)
\end{equation}
Assume that each time a subset $A$ of $\{1, \ldots, n\}$ indexes a sum we have that $2 \leq |A|+h \leq g+n-2$. Unravelling the right hand side of \ref{eq:PullbackTheta1} we have
  \begin{align*}
    \bar{\aj}_k^*[\Theta] & = - \pi_* \Bigg[\quad \; \frac{1}{2}\sum_{i=1}^n \tau_i^2D_i^2 \\
    & \qquad \qquad +  \sum_{\substack{ 0 \leq h \leq \lfloor g/2 \rfloor \\ A\subset \{1, \ldots, n\}}} \sum_{i=1}^n \tau_i\big(\gamma_h + \sum_{j \in A}\tau_j\big)D_i\delta_{h, A \cup\{n+1\}} \\
      & \qquad \qquad + \frac{1}{2}\sum_{\substack{ 0 \leq h, l \leq \lfloor g/2 \rfloor \\ A, B \subset \{1, \ldots, n\}}} \big(\gamma_h+\sum_{j \in A}\tau_j)\big(\gamma_h+\sum_{j \in B}\tau_j\big)\delta_{h, A\cup\{n+1\}} \delta_{l, B\cup\{n+1\}} \\
    & \qquad \qquad + \frac{k^2\tilde{K}^2}{2} - k\sum_{i=1}^n \tau_iD_i\tilde{K} - k \sum_{\substack{0 \leq h \leq \lfloor g/2 \rfloor \\ A \subset \{1, \ldots, n\}}} \Big(\gamma_h+ \sum_{i \in A}\tau_i\Big) \tilde{K}\delta_{h, A \cup \{n+1\}} \Bigg].
  \end{align*}
where $\gamma_h$ stands for $k(1-2h)$. Taking into account that 
  \begin{itemize}
  \item[\textbullet]
    $\tilde{K}D_i = - D_i^2$
  \item[\textbullet]
    $\pi_*(D_i\delta_{h, A\cup\{n+1\}})$ is zero unless $i\in A$, in which case it is $\delta_{h, A}$
  \item[\textbullet]
    $\pi_*(\delta_{h, A\cup\{n+1\}}\delta_{l, B\cup\{n+1\}})$ is zero unless $A = B$ and $h = l$ in which case it is equal to $-\delta_{h, A}$
  \item[\textbullet]
    $\pi_*(\tilde{K}\delta_{h, A\cup\{n+1\}}) = (2h-1)\delta_{h, A}$
  \end{itemize}
we get the expression
\begin{equation}
  \label{eq:ThetaLfinale}
  \bar{\aj}_k^*[\Theta] = \sum_{i=1}^n \left(\frac{\tau_i^2}{2} + k\tau_i\right)\psi_i - \frac{k^2}{2}\tilde{\kappa}_1 - \frac{1}{2} \sum_{\substack{0 \leq h \leq \lfloor g/2 \rfloor \\ A \subset \{1, \ldots, n\}}} \left(k(1-2h) + \sum_{i \in A} \tau_i\right)^2\delta_{h, A}
\end{equation}
where for each $A$ we assume $2 \leq |A| + h \leq g+n -2$. For $k=0$ rearranging terms we recover the relation \ref{eq:Theta2}. More generally, we recover back Hain's formula. 

\subsection{Pulling back the theta divisor in degree $g-1$}
\label{subsec:Thetag}
We briefly show how to adapt the present strategy to compute the pullback of the theta divisor on the compactification $\cJj_{g, n}^{g-1}$ of the degree $g-1$ universal jacobian with canonical trivial polarization $\sO$. A point of $\cJj_{g, n}^{g-1}$ is a tuple $(X, \underline{p}, L)$ where $(X, \underline{p})$ is a stable $n$-marked curve and $L$ is a q-stable torsion-free rank $1$ sheaf on $X$. The compactification $\cJj_{g, n}^{g-1}$ has a well defined theta divisor given by 
\begin{equation}
  \label{eq:descriptiontheta}
  \Theta_{g-1} = \{(X, L) \mid h^0(L) \geq 1\}. 
\end{equation}
See for instance \cite{MR2105707} or \cite{MR2557139}. In particular one can write down the class of the theta divisor as
\begin{equation}
  \label{eq:CThetag}
  -[\Theta_{g-1}] = \phi_*\left(\frac{c_1(\mf{L}_{g-1})^2}{2}\right) - \phi_*\left(\frac{c_1(\mf{L}_{g-1})c_1(\omega)}{2}\right)+ \lambda_1
\end{equation}
where $\phi : \cJj_{g, n\mid 1}^{g-1} \rightarrow \cJj_{g, n}^{g-1}$ is the universal curve and $\mf{L}_{g-1}$ the universal torsion-free sheaf of rank $1$ over $\cJj_{g, n \mid 1}^{g-1}$. Given an $n$-tupe of integers $\underline{\tau} = (\tau_1, \ldots, \tau_n)$ such that $\sum_i\tau_i=g-1$ we get on $\cM_{g, n}$ a section $\aj^{g-1} : \cM_{g, n} \rightarrow \cJ_{g, n}^{g-1}$ sending points $(X, \underline{p})$ on $(X, \underline{p}, \sO_X(\sum_i\tau_ip_i))$. To pullback the theta divisor on $\cJj_{g, n}^{g-1}$ it is enough to extend $\aj^{g-1}$ to codimension $1$ loci given by stable curves having $1$ node. The Abel-Jacobi map $\aj^{g-1}$ extends as it is to the case of irreducible curves. For a curve of compact type $(X, \underline{p})$ notice that a q-stable line bundle $L$ is of multidegree 
\begin{displaymath}
\deg_{X_v} L = g(v) - \chi_{p_1^c}(v)
\end{displaymath}
where $\chi_{p_1^c}(v)$ is $0$ if $p_1 \notin X_v$ and $1$ otherwise. The multidegree of $L$ does thus only depend on $(X, \underline{p})$. Twisting the line bundle $\sO(\sum_i\tau_i D_i)$ over the locus $\cMm_{g, n}^{\leq 1}$ of stable curves having at most one node we get that 
\begin{displaymath}
  \sO(\sum_i\tau_iD_i)\otimes \sO\Big(\sum_{\substack{0 \leq h \leq \lfloor g/2 \rfloor \\ A \subset \{1, \ldots, n\} \\ 2 \leq |A| \leq g+n-2}}\big(\sum_{i \in A} \tau_i - h + \chi_{A^c}(D_1)\big)\delta_{h, A\cup\{n+1\}}\Big)
\end{displaymath}
is a q-stable line bundle. It does thus define an extension $\bar{\aj}^{g-1}$ to $\cMm_{g, n}^{\leq 1}$. Pulling formula \ref{eq:CThetag} along $\bar{\aj}^{g-1}$ and using the same relations on tautological classes as the previous section we get that
\begin{equation}
  \label{eq:Thetagfinal}
  \bar{\aj}^{g-1 *}[\Theta_{g-1}]  =  \sum_{i=1}^n\frac{\tau_i(\tau_i+1)}{2}\psi_i 
                                     -\lambda _1 - \sum_{\substack{0 \leq h \leq \lfloor g/2 \rfloor\\ A \subset \{1, \ldots, n\}}} \Big(\sum_{i \in A}\tau_i - h\Big)\Big(\sum_{i \in A}\tau_i - h +1\Big)\frac{\delta_{h, A}}{2}
\end{equation}
where $2 \leq |A|+h \leq g+n-2$. Notice that this class differs from the one computed by Grushevsky and Zakharaov in \cite{GrushZakDRC} by $\delta_{irr}/8$. The reason is that the theta divisor here is not given by the vanishing of a theta function but using description \ref{eq:descriptiontheta}. As is the case in \cite{GrushZakDRC}, formula \ref{eq:Thetagfinal} generalizes a formula given by M\"uller in \cite{MR3092284}. We shortly recall their arguments. The locus $\overline{\mc{D}}_g$ computed by M\"uller is the closure in $\cMm_{g, n}$ of 
\begin{displaymath}
  \mc{D}_g = \big\{ (X, \underline{p}) \mid h^0\big(\sO_X(\sum_i \tau_ip_i)\big) \geq 1\big\}
\end{displaymath}
for smooth $X$ and at least one negative $\tau_i$. The locus $\mc{D}_g$ is given by the pullback $(\aj^{g-1})^*\Theta_{g-1}$ on $\cM_{g, n}$. Over the boundary locus $\cM_{h, A, x}\times \cM_{g-h, A^c, y}$ the Abel-Jacobi map $\bar{a}^{g-1}$ sends a curve $X = X_h\cup X_{g-h}$ on the line bundle
\begin{displaymath}
  \Big(\sO_{X_h}\big(\sum_{i \in A}\tau_iD_i - \big(\sum_{i\in A}\tau_i - h + 1\big)x\big), \sO_{X_{g-h}}\big(\sum_{i \in A^c}\tau_iD_i + \big(\sum_{i\in A}\tau_i -h)y\big)\Big).
\end{displaymath}
If any of the entries of the first factor is negative the image of $\cM_{h, A, x}\times \cM_{g-h, A^c, y}$ by $\bar{a}^{g-1}$ doesn't lie in the support of $\Theta_{g-1}$. The pullback of $\Theta_{g-1}$ over this locus is in a codimension $2$ locus of $\cMm_{g, n}$. The divisors $\bar{a}^{g-1 *}[\Theta_{g-1}]$ and $[\overline{\mc{D}}_g]$ are thus equal on each such locus. When all entries are positive $\bar{a}^{g-1}$ sends $\cM_{h, A, x} \times \cM_{g-h, A^c, y}$ into the support of $\Theta_{g-1}$. To get $[\overline{\mathcal{D}}_g]$ one needs to correct $\bar{a}^{g-1}[\Theta_{g-1}]$ by substracting the multiplicity of $\bar{a}^{g-1}$ along the previous locus. By the Riemann singularity theorem we get
\begin{displaymath}
  [\overline{\mc{D}}_g] = \bar{a}^{g-1 *}[\Theta_{g-1}] - \sum_{(h, A) \in E_+} \big(h- \sum_{i \in A} \tau_i\big)\delta_{h, A} 
\end{displaymath}  
where $E_+$ is the subset of couples $(h, A)$ having the usual constraints as in \ref{eq:Thetagfinal}, for each $i \in A$, $\tau_i$ is positive and $h \geq \sum_{i \in A}\tau_i$.

\bibliographystyle{alpha}

\end{document}